\documentclass[reqno]{amsart}
\usepackage{amssymb}
\usepackage{hyperref}

\title[Limit points for   Browder spectrum  of  operator matrices  ]
{ Limit points for   Browder spectrum   of  operator matrices  }

\author[ A. Tajmouati, M. Karmouni and S.Alaoui Chrifi ]
{  A. Tajmouati, M. Karmouni and S. Alaoui Chrifi}

\address{ A. Tajmouati, S. Alaoui Chrifi\, \newline
	Sidi Mohamed Ben Abdellah
	University,
	Faculty of Sciences Dhar Al Mahraz, Laboratory of Mathematical Analysis and Applications, Fez, Morocco.}
\email{abdelaziz.tajmouati@usmba.ac.ma}
\email{safae.alaouichrifi@usmba.ac.ma}

\address{M. Karmouni\newline
Cadi Ayyad University, Multidisciplinary Faculty, Safi, Morocco.}
\email{med89karmouni@gmail.com}

\subjclass[2010]{47A10, 47A11}
\keywords{ Browder operator, Browder spectrum, Limit point, Upper-triangular operator matrix.}

\newtheorem{theorem}{Theorem}[section]
\newtheorem{definition}{Definition}[section]

\newtheorem{lemma}{Lemma}[section]

\newtheorem{corollary}{Corollary}[section]
\newtheorem{example}{Example}

\begin{document}
\maketitle
\begin{abstract}
Let $A\in \mathcal{B}(X)$ and $B\in \mathcal{B}(Y)$, where $X$ and $Y$ are Banach spaces, and let $M_{C}$ be an operator acting on $X\oplus Y$ given by $M_C=\begin{pmatrix}
A & C \\
0 & B \\
\end{pmatrix}$. We investigate the limit point set of the Browder spectrum of $M_{C}$. It is shown that
 $$ acc \sigma_{b}(M_C)\cup W_{acc \sigma_{b}}= acc \sigma_{b}(A)\cup acc \sigma_{b}(B)$$
where $W_{acc \sigma_{b}}$ is a subsets of $ acc\sigma_{*}(B)\cap acc\sigma_{*}(A)$ and a union of certain holes in $ acc \sigma_{b}(M_C)$. Furthermore, several sufficient conditions for  $acc\sigma_{b}(M_C)=acc\sigma_{b}(A)\cup acc\sigma_{b}(B)$ holds for every $C\in \mathcal{B}(Y,X)$ are given.
\end{abstract}
%%%%%%%%%%%%%%%%%%%%%%%%%%%%%%%%%%%%%%%%%%%%%%%%%%%%%%%%%%%%%%%%%%%%%%%%%%%%%%%%%%%%%%%%%%%%%%%%
%%%%%%%%%%%%%%%%%%%%%%%%%%%%%%%%%%%%%%%%%%%%%%%%%%%%%% LEMMA %%%%%%%%%%%%%%%%%%%%%%%%%%%%%%%%%%%%%%%%%%%%%%%%%%%%%%%%%%%%%%
\section{Introduction and Preliminaries }

Throughout this paper, $X$ and $Y$ denote infinite dimensional complex Banach spaces, and $\mathcal{B}(X,Y)$ denotes the complex algebra of all bounded linear operators from $X$ to $Y$. When $Y=X$ we simply write $\mathcal{B}(X)$ instead of $\mathcal{B}(X, X)$. For $T\in \mathcal{B}(X,Y)$ we use $R(T)$ and $N(T)$ to denote the range and the null space of $T$, respectively. Write $\alpha (T)=\mbox{dim N}(T)$ and $\beta(T)=\mbox{codim R}(T)$. Sets of upper semi-Fredholm operators, lower semi-Fredholm operators, left semi-Fredholm operators and right semi-Fredholm operators are defined respectively as $\Phi_{+}(X)=\{T\in \mathcal{B}(X): \alpha(T)<\infty \mbox{ and  } R(T) \mbox{  is closed}\}$, $\Phi_{-}(X)=\{T\in \mathcal{B}(X): \beta(T) <\infty\}$,
$\Phi_{l}(X)=\{T\in \mathcal{B}(X): \alpha(T)<\infty \mbox{ and  } R(T) \mbox{ is closed and complemented subspace of X}\}$ and $\Phi_{r}(X)=\{T\in \mathcal{B}(X): \beta(T) <\infty  \mbox{ and } N(T) \mbox{ is a complemented subspace of X}\}$. Sets of semi-Fredholm and Fredholm operators are defined as $\Phi_{\pm}(X)=\Phi_{+}(X)\cup \Phi_{-}(X)$ and $\Phi(X)=\Phi_{+}(X)\cap \Phi_{-}(X)$.

The ascent of $T$ is defined as $\mbox{asc}(T)=\mbox{inf}\{n\in \mathbb{N}:N(T^n)=N(T^{n+1})\}$ and the descent of $T$ is defined as $\mbox{des}(T)=\mbox{inf}\{n\in \mathbb{N}: R(T^n)=R(T^{n+1})\}$. Where $\mbox{inf } \emptyset=\infty$. Sets of upper semi-Browder operators, lower semi-Browder operators, left semi-Browder operators and right semi-Browder operators are defined respectively $B_{+}(X)=\{T\in \Phi_{+}(X): \mbox{asc}(T)<\infty \}$, $B_{-}(X)=\{T\in \Phi_{-}(X): \mbox{dec}(T)<\infty \}$, $B_{l}(X)=\{T\in \Phi_{l}(X): \mbox{asc}(T)<\infty \}$ and $B_{r}(X)=\{T\in \Phi_{r}(X): \mbox{dec}(T)<\infty \}$. Sets of semi-Browder and Browder operators are defined as $B_{\pm}(X)=B_{+}(X)\cup B_{-}(X)$ and $B(X)=B_{+}(X)\cap B_{-}(X)$.
For $T \in \mathcal{B}(X)$, the essential spectrum, the upper semi-Browder spectrum, the lower semi-Browder spectrum, the semi-Browder spectrum, the left semi-Browder spectrum, the right semi-Browder spectrum and the Browder spectrum  are defined by:
\begin{align*}
	\sigma_{e}(T) &=\{\lambda \in \mathbb{C}: \lambda I-T\> \notin \Phi(X)\},\\
	\sigma_{b_{+}}(T) &=\{\lambda \in \mathbb{C}: \lambda I-T\> \notin B_{+}(X)\},\\
	\sigma_{b_{-}}(T) &=\{\lambda \in \mathbb{C}: \lambda I-T\> \notin B_{-}(X)\},\\
	\sigma_{b_{\pm}}(T) &=\{\lambda \in \mathbb{C}: \lambda I-T\> \notin B_{\pm}(X)\},\\
	\sigma_{lb}(T) &=\{\lambda \in \mathbb{C}: \lambda I-T\> \notin B_{l}(X)\},\\
	\sigma_{rb}(T) &=\{\lambda \in \mathbb{C}: \lambda I-T\> \notin B_{r}(X)\},\\
	\sigma_{b}(T) &=\{\lambda \in \mathbb{C}: \lambda I-T\> \notin B(X)\}.
\end{align*}
Evidently, $\sigma_{b_{+}}(T)=\sigma_{b_{-}}(T^{*})$ and $\sigma_{b_{-}}(T)=\sigma_{b_{+}}(T^{*})$, where $T^{*}\in \mathcal{B}(X^{*})$ the adjoint operator of $T$ on the dual space $X^{*}$ and $\sigma_{b}(T)=\sigma_{b_{+}}(T)\cup \sigma_{b_{-}}(T)$ consequently $\sigma_{b}(T)=\sigma_{b}(T^{*})$.

Recall that an operator $T \in \mathcal{B}(X)$ is said to have the single valued extension property at $\lambda_{0}\in \mathbb{C}$ (abbreviated SVEP) if for every open neighborhood $U\subseteq \mathbb{C}$ of $\lambda_{0}\in \mathbb{C}$, the only analytic function $f:U\rightarrow X$ which satisfies the equation $(T-\lambda I)f(\lambda)=0$ for all $\lambda\in U$ is the function $f=0$. Denote by $S(T)$ the open set of $\lambda\in \mathbb{C}$ where $T$  fails to have the SVEP at $\lambda$. An operator $T$ is said to have the SVEP if $T$ has the SVEP at every $\lambda\in \mathbb{C}$, in this case $S(T)=\emptyset$. According to \cite[Theorem 3.52]{r.10} we have
$$\sigma_{b}(T)=\sigma_{b_{+}}(T)\cup S(T^{*})= \sigma_{b_{-}}(T)\cup S(T).$$

For a compact subset $K$ of $\mathbb{C}$,  let acc$K$, int$K$, iso$K$, $K^{c}$, $\partial K$, $\overline{K}$ and $\eta(K)$ be the set of all accumulation points, the interior set, the set of isolated points, the complement, the boundary, the closure and the polynomially convex hull of $K$ respectively.

An operator $T$ is called Drazin invertible if there exists $S\in \mathcal{B}(X)$
such that
$$TS=ST,\>STS=S\>\>\mbox{and}\>\>TST-T\>\>\mbox{is nilpotent},$$
 it is well known that $S$ exists if and only if $p=asc(T)=des(T)$. The set
$\sigma_{D}(T)=\{\lambda\in\mathbb{C} : T-\lambda I \> \mbox{is not Drazin invertible} \}$ is the Drazin spectrum. Although Browder operators are exactly Fredholm Drazin invertible operators. Koliha \cite{h2} generalized the concept of Drazin invertible to generalized Drazin invertible, in fact $T \in \mathcal{B}(X)$ is generalized Drazin invertible if there exists $S\in \mathcal{B}(X)$
such that
$$TS=ST,\>STS=S\>\>\mbox{and}\>\>TST-T\>\>\mbox{is quasi-nilpotent},$$
 (i.e. $\sigma(TST-T)=\{0\}$). The former author realized that $T$ is generalized Drazin invertible if and only if $0\notin acc \sigma(T)$. The set
$\sigma_{gD}(T)=\{\lambda\in\mathbb{C} : T-\lambda I \> \mbox{is not  generalized Drazin invertible} \}=acc \sigma(T)=acc \sigma_{D}(T)$ is the generalized Drazin spectrum. For more details about those inverses we refer the reader to (\cite{Xi}, \cite{h2}, \cite{ZZJ} and \cite{ZL}). It is common knowledge that the spectrum $\sigma_{*}=\sigma_{e}, \sigma_{b},\sigma_{b_{+}}$,$\sigma_{b_{-}}$,$\sigma_{lb}$ or $\sigma_{rb}$ is a compact subset of the complex plan $\mathbb{C}$, and $\sigma_{D}(T)$ or $\sigma_{gD}(T)$ are closed subsets of $\mathbb{C}$ possibly empty.

We can easily show that
$$acc \sigma_{b}(T)=acc \sigma_{e}(T)\cup acc \sigma_{D}(T)=acc \sigma_{e}(T)\cup \sigma_{gD}(T).$$
The set $acc \sigma_{b}(T)$ may be empty, for example when $T$ is polynomially Riesz (i.e there exists a non-zero complex polynomial $P$ such that $P(T)$ is a Riesz operator). Then $P(\sigma_{b}(T))=\sigma_{b}(P(T))=\{0\}$, as a result $\sigma_{b}(T)=P^{-1}\{0\}$ is a finite set which has no accumulation point.

If $A\in\mathcal{B}(X)$, $B\in\mathcal{B}(Y)$ and $C\in\mathcal{B}(Y,X)$ then $M_C\in\mathcal{B}(X\oplus Y)$ represents a bounded linear operator on Banach space $X\oplus Y$ given by:
$$
M_C=\begin{pmatrix}
A & C \\
0 & B \\
\end{pmatrix}
$$
it called upper triangular operator matrix.
It is well know that in the case of infinite dimensional, the inclusion $\sigma(M_C)\subset\sigma(A)\cup\sigma(B)$, may be strict.
This attracts the attention of many mathematicians to study the defect ($\sigma_{*}(A)\cup\sigma_{*}(B))\setminus \sigma_{*}(M_C)$ where $\sigma_{*}$ runs over different  type of spectra.\\
 In \cite{ZHL} S.Zhang et al, gave a description of the set $\displaystyle \bigcap_{C\in\mathcal{B}(X, Y)}\sigma_b(M_C)$. They showed the following theorem which we are going to need in the sequel.

  \begin{theorem}\label{ZHL}\cite{ZHL}.
  	For given $(A, B)\in \mathcal{B}(X)\times \mathcal{B}(Y)$ the following holds:
  	
  	$$\displaystyle \bigcap_{C\in\mathcal{B}(Y, X)}\sigma_b(M_C)=\sigma_{lb}(A)\cup\sigma_{rb}(B)\cup W_{0}(A, B)$$
  	Where $W_{0}(A,B)=\{\lambda\in \mathbb{C}: N(A-\lambda)\times N(B-\lambda) \mbox{ is not isomorphic to } X/ {R(A-\lambda)}\times Y/ {R(B-\lambda)}\}.$
  \end{theorem}

  In \cite{ZZL}, authors investigate the filling-in-holes problem of $2\times2$ upper triangular operator matrices for Browder spectrum, they showed the following theorem.

   \begin{theorem}\cite{ZZL}\label{ZZL2}
   	Let $(A, B)\in \mathcal{B}(X)\times \mathcal{B}(Y)$  and  $C\in\mathcal{B}(Y, X)$. Then
   	$$\sigma_{b}(M_C)\cup W_b=\sigma_{b}(A)\cup\sigma_{b}(B)$$
   	where $W_b$ is the union of certain holes in $\sigma_{b}(M_C)$, which happen to be subsets of $\sigma_{b}(A)\cap\sigma_{b}(B)$.
   \end{theorem}

   The next lemma has been demonstrated in \cite{ZZJ}:

 \begin{lemma}\cite{ZZJ}\label{key}
 	For given $(A, B)\in \mathcal{B}(X)\times \mathcal{B}(Y)$ if $M_{C}$ is Drazin invertible for some $C\in\mathcal{B}(Y, X)$ then:
 	\begin{enumerate}
 	\item[(i)] $des(B)<\infty$ and $asc(A)<\infty.$
 	\item[(ii)] $des(A^{*})<\infty$ and $asc(B^{*})<\infty$.
 	\end{enumerate}
 	
 \end{lemma}

The purpose of this paper is to study the relationship between $acc\sigma_{b}(M_C)$ and $acc \sigma_{b}(A)\cup acc\sigma_{b}(B)$. We investigate the local spectral theory to prove the equality
$$acc \sigma_{b}(M_{C})\cup [S(A^{*})\cap S(B)] = acc \sigma_{b}(A)\cup acc \sigma_{b}(B).$$ Also,  we show that  the passage from $acc \sigma_{b}(A)\cup acc\sigma_{b}(B)$ to $acc\sigma_{b}(M_C )$ can be  described as follows: $$acc\sigma_{b}(M_C)\cup W_{acc\sigma_{b}}=acc\sigma_{b}(M_0)=acc\sigma_{b}(A)\cup acc\sigma_{b}(B)$$
where $W_{acc\sigma_{b}}$ is the union of certain holes in $acc\sigma_{b}(M_C)$, which happen to be subsets of $acc\sigma_{b}(A)\cap acc\sigma_{b}(B)$.
Finally we give sufficient conditions on $A$ and $B$ to ensure the equality $acc\sigma_{b}(M_C)=acc\sigma_{b}(A)\cup acc\sigma_{b}(B).$

\section{Main results and proofs}

%%%%%%%%%%%%%%%%%%%%%%%%%%%%%%%%%%%%%%%%%%%%%%%%%%%%%%%%%%%%%%%%%%%%%%%%%%%%%%%%%%%%%%%%%%%%%%%%%%%%%%%%%%%%%%%%%%%%
%%%%%%%%%%%%%%%%%%%%%%%%%%%%%%%%%%%%%%%  A B M_C  %%%%%%%%%%%%%%%%%%%%%%%%%%%%%%%%%%%%%%%%%%%%%%%%%%%%%%%%%%%%%%%%%%%%%%%%%%%%%%%%%%%%%%%
%%%%%%%%%%%%%%%%%%%%%%%%%%%%%%%%%%%%%%%%%%%%%%%%%%%%%%%%%%%%%%%%%%%%%%%%%%%%%%%%%%%%%%%%%%%%%%%%%%%%%%%%%%%%%%%%%%%%%%%%%%
In order to state precisely the relationship between $acc\sigma_{b}(M_C)$ and $acc\sigma_{b}(A)\cup  acc\sigma_{b}(B)$, we began this section by the following two lemmas which will be widely used in the sequel.
\begin{lemma}\label{L1}
Let $(A, B)\in \mathcal{B}(X)\times \mathcal{B}(Y)$  and  $C\in\mathcal{B}(Y, X)$. Then
$$acc\sigma_{b}(M_C)\subseteq acc\sigma_{b}(M_0)=acc\sigma_{b}(A)\cup  acc\sigma_{b}(B)$$
\end{lemma}
\begin{proof}
We have $\sigma_{b}(M_0)=\sigma_{b}(A)\cup  \sigma_{b}(B)$, it is clear  that $acc\sigma_{b}(M_0)=acc\sigma_{b}(A)\cup  acc\sigma_{b}(B)$.
  Now without lose generality let $0\notin acc\sigma_{b}(A)\cup acc\sigma_{b}(B)$ then, there exists $\varepsilon > 0$ such that $A-\lambda I$ and $B-\lambda I$ are Browder for every $\lambda$, $0<|\lambda|<\varepsilon$, according to \cite[Lemma 2.4]{DH}, $M_{c}-\lambda I$ is Browder for every $\lambda$, $0<|\lambda|<\varepsilon$. Thus $0\notin acc\sigma_{b}(M_{C})$.
 \end{proof}
 The  inclusion,  $ acc\sigma_{b}(M_C)\subset acc\sigma_{b}(A)\cup acc\sigma_{b}(B)$,  may be strict as we can see in the following example.

 \begin{example}\label{e1}
Let $A, B, C\in\mathcal{B}(l^2)$ defined by:
$$ Ae_n=e_{n+1}.$$
$$ B=A^*.$$
$$ C=e_0\otimes e_0.$$
 where $\{e_n\}_{n\in\mathbb{N}}$ is the orthonormal basis of $l^2$. We have $\sigma_{b}(A)=\{\lambda\in\mathbb{C}; |\lambda|\leq 1\}$, then $acc\sigma_{b}(A)=\{\lambda\in\mathbb{C}; |\lambda|\leq 1\}$. Since  $M_C$ is unitary,  then $acc\sigma_{b}(M_C)\subseteq\{\lambda\in\mathbb{C}; |\lambda|= 1\}$.  So  $0\notin acc\sigma_{b}(M_C)$, but $0\in acc\sigma_{b}(A)\cup acc\sigma_{b}(B)$. Notes that $A^*=B$ has not the SVEP.
\end{example}
%%%%%%%%%%%%%%%%%%%%%%%%%%%%%%%%%%%%%%%%%%%%%%%%%%%%%%%%%%%%%%%%%%%%%%%%%%%%%%%%%%%%%%%%%%%%%%%%%%%%%%%%%%%%%%%
%%%%%%%%%%%%%%%%%%%%%%%%%%%%%%%% A INVERTIBLE %%%%%%%%%%%%%%%%%%%%%%%%%%%%%%%%%%%%%%%%%%%%%%%%%%%%%%%%%%%%%%
%%%%%%%%%%%%%%%%%%%%%%%%%%%%%%%%%%%%%%%%%%%%%%%%%%%%%%%%%%%%%%%%%%%%%%%%%%%%%%%%%%%%%%%%%%%%%%%%%%%%%%%%%%%%%%%

\begin{definition}
Let $T\in\mathcal{ B}(X)$. We said that $T$ has the property $(aB)$ at $\lambda\in\mathbb{C}$ if $\lambda\notin acc\sigma_b(T)$.
\end{definition}

\begin{lemma}\label{p1p}
If two of $M_C$, $A$ and $B$ have the property $(aB)$ at $0$, then the third have also the property $(aB)$ at $0$.
\end{lemma}

\begin{proof}
\begin{enumerate}	
\item If $A$ and $B$ have the property $(aB)$ at $0$, by lemma \ref{L1} $M_C$ has the property $(aB)$ at $0$.
\item If  $M_C$ and $A$ have the property $(aB)$ at $0$, that is $0\notin acc \sigma_{b}(M_C)$ and $0\notin acc \sigma_{b}(A)$, then there exists $\varepsilon > 0$ such that $M_{C}-\lambda I$ and $A-\lambda I$ are Browder operators for every $\lambda$, $0<|\lambda|<\varepsilon$. Thus according to \cite[ Corollary 5]{hl} and \cite[Lemma 2.3]{Xi} we have $B-\lambda I$ is Browder for every $\lambda$, $0<|\lambda|<\varepsilon$, i.e $0\notin acc\sigma_{b}(B)$.
\item If $B$ and $M_C$ have the property $(aB)$ at $0$, then $A$ has the property  $(aB)$ at $0$, the proof is similar to $ii)$.
\end{enumerate}
\end{proof}
%%%%%%%%%%%%%%%%%%%%%%%%%%%%%%%%%%%%%%%%%%%%%%%%%%%%%%%%%%%%%%%%%%%%%%%%%%%%%%%%%%%%%%%%%%%%%%%%%%%%%%%%%%%%%%%%%%%%%%%%%%%
%%%%%%%%%%%%%%%%%%%%%%%%%%%%%%%% THEOREMS %%%%%%%%%%%%%%%%%%%%%%%%%%%%%%%%%%%%%%%%%%%%%%%%%%%%%%%%%%%%%%%%%%%%%%%
%%%%%%%%%%%%%%%%%%%%%%%%%%%%%%%%%%%%%%%%%%%%%%%%%%%%%%%%%%%%%%%%%%%%%%%%%%%%%%%%%%%%%%%%%%%%%
Now we are in position to prove our first main result.
\begin{theorem}\label{Th1}
	For $A\in \mathcal{B}(X)$, $B\in \mathcal{B}(Y)$ and $C\in \mathcal{B}(Y,X)$ we have
	$$acc \sigma_{b}(M_{C})\cup [S(A^{*})\cap S(B)]= acc \sigma_{b}(A)\cup acc \sigma_{b}(B).$$
\end{theorem}

\begin{proof}
It follows from lemma \ref{L1} that $acc\sigma_{b}(M_C)\subseteq acc\sigma_{b}(A)\cup  acc\sigma_{b}(B)$. Also, it is known that
$	S(A^{*})\cap S(B)  \subseteq acc\sigma(A)\cup  acc\sigma(B)
	                  \subseteq acc\sigma_{b}(A)\cup  acc\sigma_{b}(B).
	$
	Hence $acc \sigma_{b}(M_{C})\cup [S(A^{*})\cap S(B)]\subseteq acc \sigma_{b}(A)\cup acc \sigma_{b}(B)$. For the contrary inclusion, it is sufficient to prove that
	$(acc\sigma_{b}(A)\cup  acc\sigma_{b}(B))\backslash acc \sigma_{b}(M_{C}) \subseteq S(A^{*})\cap S(B)$.\\
	Let $\lambda\in (acc\sigma_{b}(A)\cup  acc\sigma_{b}(B))\backslash acc \sigma_{b}(M_{C})$ we can assume without lose of generality that $\lambda= 0$. Then $0\notin acc\sigma_{b}(M_C)$, hence there exists $\varepsilon > 0$ such that $M_{C}-\mu I$ is Browder for every $0<\arrowvert \mu \arrowvert<\varepsilon$, so $A-\mu I\in \Phi_{+}(X)$, $B-\mu I\in \Phi_{-}(Y)$ for every $0<\arrowvert \mu \arrowvert<\varepsilon$. Moreover it follows from lemma \ref{key} that $des(B-\mu I)<\infty$ and $asc(A-\mu I)<\infty$ for every $0<\arrowvert \mu \arrowvert<\varepsilon$, so $0\notin acc\sigma_{b_{+}}(A)\cup  acc\sigma_{b_{-}}(B)$. For the sake of contradiction assume that $0\notin S(A^{*})\cap S(B)$.
	
	\textbf{Case 1} $0\notin S(A^{*})$ : If $0\in \sigma_{b}(A^{*})$ we have $\sigma_{b}(A^{*})=\sigma_{b_{-}}(A^{*})\cup S(A^{*})$, then $0\in \sigma_{b_{-}}(A^{*})$. Since $0\notin acc \sigma_{b_{+}}(A)= acc \sigma_{b_{-}}(A^{*})$, it follows that $0$ is an isolated point of $\sigma_{b_{-}}(A^{*})$. On the other hand $\overline{S(A^{*})}\subseteq \sigma_{b}(A^{*})=\sigma_{b_{-}}(A^{*})\cup S(A^{*})$, thus $\partial S(A^{*})\subseteq \sigma_{b_{-}}(A^{*})$. As $\sigma_{b}(A^{*})=\sigma_{b_{-}}(A^{*})\cup S(A^{*})$ and $0\in iso \sigma_{b_{-}}(A^{*})$ then $0$ is an isolated point of $\sigma_{b}(A)=\sigma_{b}(A^{*})$. Hence $0\notin acc \sigma_{b}(A)$, according to Lemma \ref{p1p} $0\notin acc \sigma_{b}(B)$ but this is impossible. Now if $0\notin \sigma_{b}(A^{*})$ then $0\notin acc \sigma_{b}(A^{*})=acc \sigma_{b}(A)$. Hence $0\notin acc\sigma_{b}(B)$, and this is contradiction.
	
	\textbf{Case 2} $0\notin S(B)$ : If $0\in \sigma_{b}(B)$ we have $0\notin acc \sigma_{b_{-}}(B)$. Indeed $\sigma_{b}(B)=\sigma_{b_{-}}(B)\cup S(B)$ then $0\in \sigma_{b_{-}}(B)$ but $0\notin acc\sigma_{b_{-}}(B)$ thus $0\in iso \sigma_{b_{-}}(B)$ it follows that $0\in iso \sigma_{b}(B)$ i.e $0\notin acc \sigma_{b}(B)$. Since $0\notin acc \sigma_{b}(M_{C})$ then $0\notin acc \sigma_{b}(A)$, contradiction. Now if $0\notin \sigma_{b}(B)$ then $0\notin acc \sigma_{b}(B)$ since $0\notin acc \sigma_{b}(M_{C})$ thus $0\notin acc \sigma_{b}(A)$, contradiction. As a result
	$$acc \sigma_{b}(M_{C})\cup [S(A^{*})\cap S(B)] = acc \sigma_{b}(A)\cup acc \sigma_{b}(B).$$
	
	\end{proof}

From theorem \ref{Th1} , we obtain immediately the following corollary.

\begin{corollary}
Let $(A, B)\in \mathcal{B}(X)\times \mathcal{B}(Y)$ in $S(A^{*})\cap S(B)=\emptyset$, then for every $C\in\mathcal{B}(Y, X)$ we have
$$acc \sigma_{b}(M_{C})=acc \sigma_{b}(A)\cup acc \sigma_{b}(B).$$
In particular if $A^{*}$ or $B$ have he SVEP, then the last equality hold.

\end{corollary}

\begin{example}
Let $U$ be the simple unilateral shift on $l^{2}(\mathbb{N})$, set $S=U\oplus U^{*}$ the operator defined on $l^{2}(\mathbb{N})\oplus l^{2}(\mathbb{N})$. It follows that $\sigma_{b}(S)=\sigma(S)=\{\lambda\in\mathbb{C}: |\lambda|\leq 1\}$.
Consider the operators $A$, $B$ defined by $A=S+I$ and $B=S-I$, we  have
\begin{align*}
	S(A) & =\{\lambda\in\mathbb{C}: 0\leq|\lambda-1|< 1\}\\
	S(B) & = \{\lambda\in\mathbb{C}: 0\leq |\lambda+1|< 1\}
	\end{align*}
	
So,  $S(A^{*})\cap S(B)=\emptyset$. Consequently $acc \sigma_{b}(M_{C})=acc \sigma_{b}(A)\cup acc \sigma_{b}(B)$.
\end{example}

The following lemma summarizes some well-known facts which will be used frequently.

   \begin{lemma}\cite{ZL}\label{tkn}
   	Let $M$ and $N$ be two bounded subsets of complex plan $\mathbb{C}$. Then the following statements hold:
   	\begin{enumerate}
   	\item[(i)] $acc M \cup acc N = acc (M\cup N)$;
   	\item[(ii)] $iso (M\cup N)\subseteq iso M\cup iso N$;
   	\item[(iii)]  If $M$ is closed, then $\partial (acc M)\cup iso M=\partial M$;
   	\item[(iv)] If $M$ is closed, then $iso(\partial M)=iso M$.
   	\end{enumerate}
   \end{lemma}

 The following two lemmas are the key of our second main result.
 \begin{lemma}\label{l3}
 	Let $E$, $F$ be two compact subsets of $\mathbb{C}$, such that $E\subseteq F$ and $\partial F\subseteq E$, then
 	$$\partial acc F\subseteq acc E.$$
 \end{lemma}

 \begin{proof}
 	For any $\lambda \in \partial (accF)$ then either $\lambda \in acc(\partial acc F)$ or $\lambda \in iso(\partial acc F)$.\\
 	\textbf{Case 1:} If $\lambda \in acc(\partial acc F)$, then  there exist $\lambda_{n}\in \partial acc F$ $n=1,2,...$ such that $\displaystyle\lim_{n\rightarrow \infty} \lambda_{n}=\lambda $, since
 	$$\partial acc F\subseteq \partial F \subseteq E$$
 	it follows that $\lambda_{n}\in E$, $n=1,2,...$,  thus we have $\lambda\in acc E$.\\
 	\textbf{Case 2:} If $\lambda \in iso(\partial acc F)$, then we get $\lambda \in iso(acc F)$ from Lemma \ref{tkn}, that is
 	$$iso(\partial acc F) = iso(acc F).$$
 	Thus there exists $\varepsilon > 0$ such that $\tau \notin acc F$ for every $\tau$, $0<|\lambda-\tau|<\varepsilon$, but $\lambda \in acc F$ then there exist $\mu_{n}\in \partial F$ $n=1,2,...$ such that $\displaystyle\lim_{n\rightarrow \infty} \mu_{n}=\lambda $ and $\mu_{n}\neq \lambda$ for every $n=1,2,...$
 	i.e there exists $N\in\mathbb{N}^{*}$ such that $0<|\mu_{n}-\lambda|<\varepsilon$ for every $n\geq N$.\\
 	Now let $\lambda_{n}=\mu_{N+1+n}$, then $\lambda_{n}\in iso F$ $n=1,2,...$ and $\displaystyle\lim_{n\rightarrow \infty} \lambda_{n}=\lambda $. It follows from Lemma \ref{tkn} that $$iso F\subseteq \partial F \subseteq E$$
 	so $\lambda_{n}\in iso F \subseteq E$ for every $n=1,2,...$ and $lim_{n\rightarrow \infty} \lambda_{n}=\lambda $, therefore $\lambda\in acc E$.
 	Then for both cases $\partial acc F\subseteq acc E$ is true.
 \end{proof}

 \begin{lemma}\label{l4}
 	Let $E$, $F$ be two compact subsets of $\mathbb{C}$ such that $E\subseteq F$ and $\eta(E)=\eta(F)$ then $\eta(acc E)=\eta(acc F)$.
 \end{lemma}

 \begin{proof}
 	Since $E\subseteq F$ then $acc E\subseteq acc F$, we need to show that $\partial (acc F)\subseteq \partial (acc E)$ from the maximum module theorem. But since $int (acc E)\subseteq int (acc F)$ it suffices to show that $\partial (acc E)\subseteq (acc F)$ which is always verified by lemma \ref{l3}.
 \end{proof}
 The following theorem says that the passage from $acc \sigma_{b}(A) \cup acc \sigma_{b}(B)$ to $acc \sigma_{b}(M_{C})$ is the punching of some set in $acc \sigma_{b}(A) \cap acc \sigma_{b}(B)$.
 \begin{theorem}\label{th2}
 	Let $A\in \mathcal{B}(X)$, $B\in \mathcal{B}(Y)$ and $C\in \mathcal{B}(Y,X)$. Then
 	$$acc \sigma_{b}(M_{C})\cup W_{acc\sigma_{b}}= acc \sigma_{b}(A) \cup acc \sigma_{b}(B)$$
 	where $W_{acc\sigma_{b}}$ is the union of certain holes in $acc \sigma_{b}(M_{C})$, which happen to be subsets of $acc \sigma_{b}(A) \cap acc \sigma_{b}(B)$.
 \end{theorem}

 \begin{proof}
 	We first claim that, for every $C\in \mathcal{B}(Y,X)$ we have
 	\begin{equation}\label{eq1}
 	(acc \sigma_{b}(A) \cup acc \sigma_{b}(B))\backslash (acc \sigma_{b}(A) \cap acc \sigma_{b}(B))\subseteq acc \sigma_{b}(M_{C}),
 	\end{equation}
 	 to see this suppose that $\lambda \in (acc \sigma_{b}(A) \cup acc \sigma_{b}(B))\backslash (acc \sigma_{b}(A) \cap acc \sigma_{b}(B))$ then either $\lambda\in acc \sigma_{b}(A) \backslash acc \sigma_{b}(B)$ or $\lambda\in acc \sigma_{b}(B) \backslash acc \sigma_{b}(A)$.
 	 \begin{enumerate}
 	\item If $\lambda\in acc \sigma_{b}(A) \backslash acc \sigma_{b}(B)$ it follows that $A$ has not the property $(aB)$ at $\lambda$ and $B$ has the property $(aB)$ at $\lambda$, thus $\lambda \in acc \sigma_{b}(M_{C})$, for if it were not so, then $\lambda \notin acc \sigma_{b}(M_{C})$ and hence according to lemma\ref{p1p}
 	$\lambda \notin acc \sigma_{b}(A)$ which is impossible.
 	\item If $\lambda\in acc \sigma_{b}(B) \backslash acc \sigma_{b}(A)$, similarly as in (1) we can  show that $\lambda \in acc \sigma_{b}(M_{C})$.
 \end{enumerate}
 	Moreover, from the proof of Theorem \ref{ZZL2}) we have
 	$$\partial(\sigma_{b}(A) \cap \sigma_{b}(B))\subseteq \sigma_{b}(M_{C}).$$ Applying lemma \ref{l3} and lemma \ref{l4} we have
 	\begin{equation}\label{eq2}
 	\eta (acc \sigma_{b}(M_{C}))=\eta (acc \sigma_{b}(A) \cup acc \sigma_{b}(B)).
 	\end{equation}
 	Therefore \ref{eq2} says that the passage from $acc \sigma_{b}(M_{C})$ to $acc \sigma_{b}(A) \cup acc \sigma_{b}(B)$ is the filling in certain of the holes in $acc \sigma_{b}(M_{C})$. But $(acc \sigma_{b}(A) \cup acc \sigma_{b}(B))\backslash acc \sigma_{b}(M_{C})$ is contained in $acc \sigma_{b}(A) \cap acc \sigma_{b}(B)$ by \ref{eq1}, it follows that the filling in certain of the holes in $ acc \sigma_{b}(M_{C})$ should occur in $acc \sigma_{b}(A) \cap acc \sigma_{b}(B)$.
 \end{proof}

%%%%%%%%%%%%%%%%%%%%%%%%%%%%%%%%%%%%%%%%%%%%%%%%%%%%%%%%%%%%%%%%%%%%%%%%%%%%%%%%%
%%%%%%%%%%%%%%%%%%%%   %%%%%%%%%%%%%%%%%%%%%%%%%%%%%
%%%%%%%                       %%%%%%%%%%%%%%%%%%%%%%%%%%%%%%%%%%%%%%%%%%%          %%%%%%%%%%
\begin{corollary}
Let $(A, B)\in \mathcal{B}(X)\times \mathcal{B}(Y)$. If $ acc\sigma_{b}(A)\cap acc \sigma_{b}(B)$ has no interior points, then  for every  $C\in\mathcal{B}(Y, X)$ we have $$acc\sigma_{b}(M_C)=acc\sigma_{b}(A)\cup acc\sigma_{b}(B).$$
\end{corollary}
%%%%%%%%%%%%%%%%%%%%%%%%%%%%%%%%%%%%%%%%%%%%%%%%%%%%%%%
%%%%%%%%%%%%%%%%%%%%%%%%%%%%%%%%%%%%%%%%%%%%%%%%%%%%%%%%%%%%%%%%%%%%%%%%%%%%%%
%%%%%%%%%%%%%%%%%%%%%%%%%%%%%%%%%%%%%%%%%%%%%%%%%%%%%%%%%%%%%%%%%%%%%%%%%%%%%%%%%%%%%%%%%%%%%

 From theorem \ref{ZZL2} and theorem \ref{th2}, we have the following.

\begin{theorem}\label{tthh2}
	Let $(A, B)\in \mathcal{B}(X)\times \mathcal{B}(Y)$  and  $C\in\mathcal{B}(Y, X)$. Then the following assertions are equivalent
	\begin{enumerate}
		\item[(i)] $\sigma_{b}(M_C)=\sigma_{b}(A)\cup\sigma_{b}(B)$,
		\item[(ii)] $acc\sigma_{b}(M_C)=acc\sigma_{b}(A)\cup acc\sigma_{b}(B).$
	\end{enumerate}
	
\end{theorem}

\begin{proof}
	First we show that \begin{equation}\label{inc}
	W_{b}\subseteq W_{acc \sigma_{b}}.
	\end{equation}
	Indeed, if $\lambda\in W_{b}$ then according to Theorem \ref{ZZL2} we have $\lambda\in (\sigma_{b}(A)\cup \sigma_{b}(B))\backslash \sigma_{b}(M_{c})$, then $\lambda \notin \sigma_{b}(M_{c})$ consequently $\lambda \notin acc \sigma_{b}(M_{c})$, it is enough to show that $\lambda\in acc(\sigma_{b}(A)\cup \sigma_{b}(B))$, if it was not then $\lambda\notin acc(\sigma_{b}(A)\cup \sigma_{b}(B))$ but $\lambda\in \sigma_{b}(A)\cup\sigma_{b}(B)$ thus
	\begin{align*}
	\lambda\in iso(\sigma_{b}(A)\cup\sigma_{b}(B)) & \subseteq iso(\sigma_{b}(A))\cup iso(\sigma_{b}(B))\>\>\>\> \mbox{(Lemma \ref{tkn} )} \\
	& \subseteq \partial \sigma_{b}(A)\cup \partial \sigma_{b}(B) \>\>\>\> \mbox{(Lemma \ref{tkn} )}\\
	& \subseteq \sigma_{lb}(A)\cup\sigma_{rb}(B)\>\>\>\> \mbox{(Theorem \ref{ZHL} )}\\
	& \subseteq \sigma_{b}(M_C)
\end{align*}

	hence $\lambda \in \sigma_{b}(M_C)$, contradiction. Therefore
	$$\lambda\in (acc\sigma_{b}(A)\cup acc\sigma_{b}(B))\backslash acc \sigma_{b}(M_{c})$$
	by theorem \ref{th2} we have $\lambda\in W_{acc\sigma_{b}}$, so $W_{b}\subseteq W_{acc\sigma_{b}}$, according to this inclusion the following implication is hold $$ acc\sigma_{b}(M_C)=acc\sigma_{b}(A)\cup acc\sigma_{b}(B) \Rightarrow \sigma_{b}(M_C)=\sigma_{b}(A)\cup\sigma_{b}(B)$$
	Conversely, if $\sigma_{b}(M_C)=\sigma_{b}(A)\cup\sigma_{b}(B)$ then
	\begin{align*}
		acc (\sigma_{b}(M_C)) &= acc(\sigma_{b}(A)\cup\sigma_{b}(B))\\
		&= acc(\sigma_{b}(A))\cup acc(\sigma_{b}(B)).
	\end{align*}

\end{proof}
The following example shows that the inclusion \ref{inc} used in the proof of theorem \ref{tthh2} may be strict in general.
\begin{example}
Define $U,V\in \mathcal{B}(l^{2})$ by

	$$U e_{n} =e_{n+1}$$
	$$V e_{n+1}= e_{n}$$
	where $\{e_{n}\}_{n\in \mathbb{N}}$ is the orthonormal basis of $l^{2}$. Let us introduce an operator $P:l^{2}\rightarrow l^{2}$ as:
	$$P(x_{1},x_{2},x_{3},...)=(x_{1},0,0,...), \>\>\>\>\>\>\>\>\>\> (x_{1},x_{2},x_{3},...)\in l^{2}.$$
	consider the operator $
	M_C=\begin{pmatrix}
	A & C \\
	0 & B \\
	\end{pmatrix}
	: l^{2}\oplus l^{2}\oplus l^{2}\rightarrow l^{2}\oplus l^{2}\oplus l^{2}
	$, where $A=U$, $B=\begin{pmatrix}
	V & 0 \\
	0 & 0 \\
	\end{pmatrix}: l^{2}\oplus l^{2}\rightarrow l^{2}\oplus l^{2}$ and $C=(P,0): l^{2}\oplus l^{2}\rightarrow l^{2}$.\\
	We have $\sigma_{b}(M_C)=\{\lambda\in\mathbb{C}: |\lambda|=1\}\cup\{0\}$, $\sigma_{b}(A)=\sigma_{b}(B)=\{\lambda\in\mathbb{C}: |\lambda|\leq 1\}$; then $acc \sigma_{b}(M_C)=\{\lambda\in\mathbb{C}: |\lambda|=1\}$, $acc \sigma_{b}(A)=acc \sigma_{b}(B)=\{\lambda\in\mathbb{C}: |\lambda|\leq 1\}$. Consequently
	$$W_{\sigma_{b}}=\{\lambda\in\mathbb{C}: 0<|\lambda|< 1\} \>,\> W_{acc\sigma_{b}}=\{\lambda\in\mathbb{C}: |\lambda|< 1\}.$$
	Thus $W_{\sigma_{b}}\neq W_{acc\sigma_{b}}$.
	
\end{example}
Nevertheless, we have the following theorem.
\begin{theorem}
	Let $(A, B)\in \mathcal{B}(X)\times \mathcal{B}(Y)$  and  $C\in\mathcal{B}(Y, X)$. If $iso \partial W_{\sigma_{b}}=\emptyset$ then $$W_{\sigma_{b}}=W_{acc \sigma_{b}}.$$
\end{theorem}

\begin{proof}
	According to Lemma \ref{tkn} , and since $iso \partial W_{\sigma_{b}}=\emptyset$, we obtain
	$$iso \sigma_{b}(M_C)=iso(\sigma_{b}(A)\cup \sigma_{b}(B))\subseteq iso\sigma_{b}(A)\cup iso \sigma_{b}(B).$$
	Let $\lambda\in iso \sigma_{b}(M_C)$, then either $\lambda\in iso \sigma_{b}(A)$ or $\lambda\in iso \sigma_{b}(B)$. If $\lambda\in iso \sigma_{b}(A)$, then $A$ has the property $(aB)$ at $\lambda$ and $\lambda I-A$ is not Browder, by Lemma \ref{p1p} , B has also the property $(aB)$ at $\lambda$, thus $\lambda\in iso \sigma_{b}(B)$ or $\lambda\in\rho_{b}(B)$. In contrast, if $\lambda \in iso \sigma_{b}(B)$ we obtain similarly that $\lambda\in iso \sigma_{b}(A)$ or $\lambda\in\rho_{b}(A)$. this means that
	$$iso \sigma_{b}(M_C)\subseteq (iso \sigma_{b}(A)\cap iso \sigma_{b}(B))\cup(iso \sigma_{b}(A)\cap \rho_{b}(B))\cup (\rho_{b}(A)\cap iso \sigma_{b}(B)).$$
	In addition, due to theorem \ref{ZHL} and Lemma \ref{tkn} , we have
	\begin{align*}
		(iso \sigma_{b}(A)\cap iso \sigma_{b}(B))\cup & (iso \sigma_{b}(A)\cap \rho_{b}(B))\cup (\rho_{b}(A)\cap iso \sigma_{b}(B))\\
		 & \subseteq iso \sigma_{b}(A)\cup iso \sigma_{b}(B) \subseteq \partial\sigma_{b}(A)\cup \partial \sigma_{b}(B)\\
		 & \subseteq \sigma_{lb}(A)\cup \sigma_{rb}(B)\subseteq \sigma_{b}(M_C),
		\end{align*}
		and from Lemma \ref{p1p}, we have
		$$ (iso \sigma_{b}(A)\cap iso \sigma_{b}(B))\cup(iso \sigma_{b}(A)\cap \rho_{b}(B))\cup (\rho_{b}(A)\cap iso \sigma_{b}(B))\subseteq acc \sigma_{b}(M_C)^{c},$$
		it follows that $$ (iso \sigma_{b}(A)\cap iso \sigma_{b}(B))\cup(iso \sigma_{b}(A)\cap \rho_{b}(B))\cup (\rho_{b}(A)\cap iso \sigma_{b}(B))\subseteq iso \sigma_{b}(M_C).$$
		Consequently, $ (iso \sigma_{b}(A)\cap iso \sigma_{b}(B))\cup(iso \sigma_{b}(A)\cap \rho_{b}(B))\cup (\rho_{b}(A)\cap iso \sigma_{b}(B))= iso \sigma_{b}(M_C)$, but $iso \sigma_{b}(M_C)\cap(acc \sigma_{b}(A)\cup acc \sigma_{b}(B))=\emptyset$.\\
		From Lemma \ref{p1p}, it follows that
		\begin{align*}
			(iso \sigma_{b}(A)\cup iso \sigma_{b}(B))\backslash iso \sigma_{b}(M_C) &= (iso \sigma_{b}(A)\backslash iso \sigma_{b}(M_C))\cup (iso \sigma_{b}(B)\backslash iso \sigma_{b}(M_C))\\
			&\subseteq acc \sigma_{b}(B)\cup acc \sigma_{b}(A).
			\end{align*}
			These imply that
			\begin{align*}
				\sigma_{b}(A)\cup \sigma_{b}(B) &= (acc \sigma_{b}(A)\cup acc \sigma_{b}(B))\cup (iso \sigma_{b}(A)\cup iso \sigma_{b}(B))\\
				&= acc \sigma_{b}(A)\cup acc \sigma_{b}(B)\cup iso \sigma_{b}(M_C) \cup (iso \sigma_{b}(A)\cup iso \sigma_{b}(B))\backslash iso \sigma_{b}(M_C) \\
				& \subseteq (acc \sigma_{b}(A)\cup acc \sigma_{b}(B))\cup iso \sigma_{b}(M_C)\\
				& \subseteq acc \sigma_{b}(M_C) \cup W_{acc \sigma_{b}} \cup iso \sigma_{b}(M_C).
				\end{align*}
				However
				\begin{align*}
					\sigma_{b}(M_C)\cap W_{acc \sigma_{b}} &= (acc \sigma_{b}(M_C)\cup iso \sigma_{b}(M_C))\cap W_{acc \sigma_{b}}\\
					&= (acc \sigma_{b}(M_C)\cap W_{acc \sigma_{b}})\cup (iso \sigma_{b}(M_C)\cap W_{acc \sigma_{b}})\\
					&= \emptyset.
					\end{align*}
					On the other hand, $\sigma_{b}(A)\cup \sigma_{b}(B)=\sigma_{b}(M_C)\cup W_{\sigma_{b}}=(acc \sigma_{b}(M_C)\cup iso \sigma_{b}(M_C)) \cup W_{ \sigma_{b}}$ and $\sigma_{b}(M_C)\cap W_{ \sigma_{b}}=\emptyset$, this highlight that $W_{ \sigma_{b}}=W_{acc\sigma_{b}}$.
	
	\end{proof}
%%%%%%%%%%%%%%%%%%%%%%%%%%%%%%%%%%%%%%%%%%%%%%%%%%%%%%%%%%%%%%%%%%%%%%%%%%%%%%%%%%%%
%%%%%%%%%%%%%%%%%%%%%%%%%%%%%%%%%%%%%%%%%%%%%%%%%%%%%%%%%%%%%%%%%%%%%%%%%%%%%%%%%%%%%%%%%%%%%
%%%%%%%%%%%%%%%%%%%%%%%%%%%%%%%%%%%%%%%%%%%%%%%%%%%%%%%%%%%%%%%%%%%%%%%%%%%%%%%%%%%%%%%%%%%%%%%%%%%%%%%%%
%%%%%%%%%%%%%%%%%%%%%%


\begin{thebibliography}{99}
\bibitem{r.10}   \textsc{ Aiena P. }\emph{Fredholm and Local Spectral Theory, with Application to Multipliers,} Kluwer Academic, 2004.

\bibitem{BZZ}
\textsc{Benhida C,  Zerouali E H,  Zguitti H. } \emph{ Spectra of upper triangular operator matrices, }  Proc. Am. Math. Soc.
Vol 133, No.10, (2005),  3013-3020.

\bibitem{Xi}
\textsc{X.H. Cao, M.Z. Guo, B. Meng.} \emph{ Drazin spectrum and Weyl's theorem for operator matrices, } J. Math. Res. Exposition,
Vol 26, No.3 (2006), 413-422.


\bibitem{DH}
\textsc{Djordjevic S V, Han Y M. } \emph{ A note of Weyl's Theorem for operator matrices, }  Proc. Amer. Math. Soc,
Vol 131 (2002), 2543-2547.

\bibitem{dp}
\textsc {Du H K, Pan J. } \emph{Perturbation of spectrum of $2 \times 2$ operator matrices, } Proc. Am. Math. Soc. Vol 121  (1994),  761-766.



\bibitem{hl}
\textsc{Han J K, Lee H Y, Lee W Y. }  \emph{ Invertible completions of $2\times 2$  upper triangular operator matrices, }  Proc. Am. Math. Soc. Vol 128, No.1 (2000),  119-123.

\bibitem{h2}
\textsc{J. J. Koliha. }  \emph{A generalized Drazin inverse, } Glasgow Math.J. Vol 38 (1996),  367-381.

\bibitem{lau}  \textsc{Laursen K B, Neumann M M. }
\emph{An introduction to Local Spectral Theory. }  London Mathematical Society Monograph, New series, Vol 20, Clarendon Press, Oxford, 2000.

\bibitem{Lee}
\textsc {W.Y. Lee. } \emph{Weyl spectra of operator, } Proc. Am. Math. Soc. Vol 129 (2001),  131-138.


\bibitem{Per}
\textsc {C. Pearcy. } \emph{Topics in Operator Theory, } CBMS Reg. Conf. Ser. Math,  Amer. Math. Soc, vol 36, 1978.



\bibitem{ZZA}
\textsc{H. Zariouh, H. Zguitti. } \emph{On pseudo B-Weyl operators and generalized drazin invertible for operator matrices, }
Linear and Multilinear
Algebra, Vol 64 (2016),  1245-1257.

\bibitem{ZZZ}
\textsc{Zerouali E H,  Zguitti H. } \emph{Perturbation of spectra of operator matrices and local spectral theory, }
J. Math.  Anal.  Appl, Vol 324 (2006),  992-1005.


\bibitem{ZZJ}
\textsc{Zhang S. F, Zhong H. J, Jiang Q. F. } \emph{Drazin spectrum of operator matrices on the Banach space, }
Linear Algebra Appl, Vol 429 (2008),  2067-2075.

\bibitem{ZL}
\textsc{ Zhang S, Zhong H, Lin L. } \emph{Generalized Drazin spectrum of operator matrices,} Appl. Math. J.
Chin. Univ, Vol. 29 (2014),  162-170.

\bibitem{ZHL}
\textsc{ Zhang S, Lin L, Zhong H. } \emph{Perturbation of Browder spectrum of upper triangular operator matrices,} Linear Algebra Appl, Vol. 64 (2016) 502-511.


\bibitem{ZZL}
\textsc{ Zhang Y N,  Zhang H J,  Lim L Q. } \emph{ Browder Spectra and Essential Spectra of Operator Matrices, } Acta Mathematica Sinica, English Series
Vol. 24, No.6  (2008),  947-954.

\bibitem{ZZW}
\textsc{ Zhang, S.F., Zhang, H.J., Wu, J. } \emph{ Spectra of Upper-triangular Operator Matrix, } Acta Math. Sin. (in
Chinese)
Vol. 54 (2011),  41-60.

\end{thebibliography}
\end{document}